\documentclass[12pt,centertags,oneside]{amsart}
\usepackage{amsmath,amstext,amsthm,amscd,typearea,hyperref}
\usepackage{amssymb}
\usepackage{a4wide}
\usepackage[mathscr]{eucal}
\usepackage{mathrsfs}
\usepackage{typearea}
\usepackage{charter}
\usepackage{pdfsync}
\usepackage[a4paper,width=16.2cm,top=3cm,bottom=3cm]{geometry}

\numberwithin{equation}{section}

\newtheorem{theorem}{Theorem}[section]

\newtheorem{proposition}[theorem]{Proposition}
\newtheorem{corollary}[theorem]{Corollary}

\newtheorem{remark}[theorem]{Remark}

\newcommand{\cali}[1]{\mathscr{#1}}

\newcommand{\End}{{\rm End}}

\newcommand{\supp}{{\rm supp}}
\newcommand{\wi}{\widetilde}

\newcommand{\ov}{\overline}

\newcommand{\dist}{\mathop{\mathrm{dist}}\nolimits}

\newcommand{\ddc}{dd^c}
\newcommand{\dc}{d^c}

\def\Ker{\operatorname{Ker}}

\def\mP{\mathcal{P}}
\def\mQ{\mathcal{Q}}

\DeclareMathOperator{\spec}{Spec}

\newcommand{\dbar}{\overline\partial}

\newcommand{\Cc}{\cali{C}}

\newcommand{\Lc}{\cali{L}}

\newcommand{\Sc}{\cali{S}}

\newcommand{\C}{\mathbb{C}}

\newcommand{\N}{\mathbb{N}}
\newcommand{\Z}{\mathbb{Z}}
\newcommand{\R}{\mathbb{R}}

\newcommand{\norm}[1]{\lVert#1\rVert}

\newcommand{\boldsym}[1]{\boldsymbol{#1}}

\newcommand\bk{\boldsym{k}}

\newcommand\br{\boldsym{r}}

\newcommand{\comment}[1]{}

\title[Asymptotic behavior of Bergman kernels]
{On the asymptotic behavior of Bergman kernels for
\break positive line bundles}

\author{Tien-Cuong Dinh}
\address{Department of Mathematics, National University 
of Singapore, 10 Lower Kent Ridge Road, Singapore 119076.}
\email{matdtc@nus.edu.sg}
\author{Xiaonan Ma}
\address{Universit\'e Paris Diderot - Paris 7,
UFR de Math\'ematiques, Case 7012,
75205 Paris Cedex 13, France.}
\email{xiaonan.ma@imj-prg.fr}
\author{Vi{\^e}t-Anh Nguy{\^e}n}
\address{Math{\'e}matique-B{\^a}timent 425, UMR 8628, 
Universit{\'e} Paris-Sud, 91405 Orsay, France. }
\email{VietAnh.Nguyen@math.u-psud.fr}

\begin{document}

\begin{abstract} 
Let $L$ be a  positive line bundle on a projective complex manifold.
 We study the asymptotic behavior of Bergman kernels
 associated with the tensor powers $L^p$ of $L$ as $p$ tends 
 to infinity. The emphasis is the dependence of the uniform estimates
 on the positivity of the Chern form of the metric on $L$.
This situation appears naturally when we approximate a semi-positive 
singular metric by smooth positively curved metrics.
\end{abstract}

\maketitle

\medskip

\noindent
{\bf Classification AMS 2010}: 32U15 (32L05).

\medskip


\section{Introduction} \label{introduction}

Let $L$ be an ample holomorphic line bundle  over 
a projective manifold $X$ 
of dimension $n$. Fix a (reference) smooth Hermitian metric $h_0$ 
on $L$ 
whose first Chern
form $\omega_0$ is a K\"ahler form. Recall that 
$\omega_0= \frac{\sqrt{-1}}{2\pi}R^L_{0}$, 
where $R^L_{0}$ is the curvature of the Chern connection on 
$(L,h_{0})$. 

Let $h^L$ be a semi-positive singular 
metric on $L$.  For various applications,
one needs to understand the asymptotic behavior of 
the Bergman kernel associated with $L^{p}$ and $h^L$
when $p$ tends to infinity. 
A natural approach is to approximate the considered metric by 
smooth positively curved metrics, 
and therefore, it is necessary to understand the dependence 
of the Bergman kernels
 in terms of the positivity of the curvature of the metric. 
 See \cite{BlockiKolodziej, Demailly92, DMN15} for the regularization 
 of metrics. This method was already used in our previous work
 on the speed of convergence of Fekete points, 
 see \cite{BBW, DMN15}.
 In  \cite[\S 2.3]{DMN15},
 inspired by \cite{Berndtsson}, an $L^1$-estimate for Bergman 
 kernels was obtained. Here, we investigate the uniform estimate
 which can be useful for applications in geometry.  
 
 Fix a smooth K\"ahler form $\theta$ on $X$ (one can take 
 $\theta=\omega_0$). 
Consider a metric $h=e^{-2\phi}h_0$ on $L$ with  weight $\phi$ 
of class $\Cc^{n+6}$ 
whose first Chern form $\omega:=\ddc\phi+\omega_0$
(here $\dc:= {\sqrt{-1}\over 2\pi }(\overline\partial -\partial)$)
satisfies 
 \begin{equation}\label{eq_strict_positivity}
\omega\geq \zeta\theta \quad \text{for some constant } 
0<\zeta\leq 1.
\end{equation}
Consider the natural metric on the space of smooth sections 
of $L^{p}$, induced by the metric $h$ on $L$ and 
the volume form $\theta^n$ on $X$, which is defined by
\begin{align}\label{eq:n2.2}
\|s\|^2_{L^2(p\phi)}:=\int_X |s(x)|_{p\phi}^2 \theta^n/n!.
\end{align}
Here, $|s(x)|_{p\phi}$ stands for the norm of $s(x)$ with respect 
to the metric $h^{\otimes p}$ on $L^p$. 
Let $\langle\cdot,\cdot\rangle_{p\phi}$ be the associated Hermitian  
product on $\Cc^\infty(X,L^{p})$, 
the space of smooth sections of $L^p$.
Let $P_p$ be the orthogonal projection from 
$(\Cc^\infty(X,L^{p}), \langle\cdot,\cdot\rangle_{p\phi})$
onto the subspace of holomorphic sections $H^0(X,L^{p})$. 
The Bergman kernel associated with the above data is the kernel 
associated with the last projection where we use
the volume form $\theta^n/n!$ to integrate functions on $X$. 
This kernel is denoted by 
$P_p(x,x')$, with $x,x'\in X$. It is a section of 
the line bundle over $X\times X$ which is the tensor product 
of two line bundles: the first one is the pull-back to $X\times X$ 
of the line bundle $L^{p}$
over the first factor, and the second one is the pull-back of
the dual line bundle $(L^*)^{p}$ of $L^{p}$
over the second factor.
In particular, its restriction to the diagonal of $X\times X$, i.e.,
$P_p(x,x)$,  can be identified to a positive-valued function on $X$. 
See \cite{MaMarinescu07} for details.
In fact, if $\{s_{j}\}_{j}$ is an orthonormal basis of
$(H^0(X,L^{p}), \langle\cdot,\cdot\rangle)$, then
\begin{align}\label{eq:n2.3}
P_{p}(x,x)= \sum_{j}|s_{j}(x)|_{p\phi}^2
=\sup \big\{ |s(x)|^2_{p\phi},\quad s\in H^0(X,L^p) 
\text{ with } \|s\|_{L^2(p\phi)}=1 \big\}.
\end{align}

Here is the main result in this paper which gives us a uniform 
estimate of 
the Bergman kernel in terms of $\phi,\omega,p$ and $\zeta$.
This is a version of Tian's theorem \cite{Tian}.
See  \cite{Berndtsson, BoS, Catlin, ComanMarinescu, 
DaiLiuMa06, Hs, MaMarinescu14, Xu12, Zelditch} for various generalizations.
We also refer to \cite{MaMarinescu07} for a comprehensive 
study of several analytic and geometric aspects of Bergman kernels.
The last reference is 
inspired by the analytic localization technique in \cite{BL}.

\begin{theorem}\label{tue4} 
Under the above assumptions, 
there exist $\delta>0$, $c>0$ satisfying the following condition:  
for any $l\in \N^*$, there is a constant $c_l>0$ 
such that for  $p\in \N^*$, 
$p \zeta>\delta$, and $x\in X$, we have
\begin{align}\label{smp1.6}
\begin{split}
\Big|p^{-n}P_{p}(x,x)- {\omega(x)^n\over\theta(x)^n}\Big| 
\leqslant & c\, |d\phi|_{{n+5}}^{2n+8}|\omega|_{0}^{4n+20}
|d\phi|_{n+2}^{2n+2} \zeta^{-2n-10} p^{-1}  \\
& +c_{l}|\omega|_{{n}}^{2n+2} (|d\phi|_{2}\zeta^{-1})^{6n+6+3l}
p^{-l} .
\end{split}\end{align}
\end{theorem}

Note that $|\cdot|_k$ stands for $1+\|\cdot\|_{\Cc^k}$. 
As a direct consequence, we  infer the following result by taking 
$l=1$.
  
 \begin{corollary} \label{cor_order_rho} 
     There exist $\delta>0, c>0$ such that for any
     $0<\zeta\leq 1$, 
 any weight $\phi$ of class $\Cc^{n+6}$ with
 $\ddc \phi+\omega_0\geq  \zeta \theta$,
and any $p\in \N^*$ with $\zeta p>\delta$, we have 
 \begin{align}\label{eq:n2.4}
\Big|p^{-n}P_{p}(x,x) - {\omega(x)^n\over\theta(x)^n} \Big| 
\leq   c\, \zeta^{-6n-9} |d\phi|^{8n+30}_{n+5} p^{-1}.
\end{align}
 \end{corollary}

   If  $\phi\in \Cc^{n+2k+6}$,
   we can adapt  easily 
   the proof of Theorem \ref{tue4} to get 
   the estimate for $\Cc^k$-norm
     of the left hand side  of \eqref{smp1.6}.
     Cf. Remark \ref{Rem3.10}.

The article is organized as follows. 
In Section \ref{bsys3.3}, we reduce the problem to the local setting.
In Section \ref{s3.3}, we establish Theorem \ref{tue4}. 
We need an approach different from previous ones which use 
the normal coordinates and the extension of connections on $L$, 
see \cite[\S 4.2]{DaiLiuMa06} and \cite[\S 4.1.3]{MaMarinescu07}. 
Note that throughout the paper, the constants $c,c', c_l, \ldots $ 
may be changed from line to line.

  \medskip\noindent
{\bf Acknowledgment.} 
T.-C.\ D.\  supported by Start-Up 
Grant R-146-000-204-133 from National University of Singapore.
The paper was partially written during the visits of the second and 
third authors at National University of Singapore. 
They would like to thank this university for its support and hospitality.

\section{Localization of the problem}\label{bsys3.3}

Recall that the complex structure on $X$ is given by a smooth 
section $J$ of the vector bundle $\End(TX)$ such that $-J^2$
is the identity section. Here, $TX$ denotes the real tangent bundle 
of $X$. Denote also by 
$T^{(1,0)}X$ and $T^{(0,1)}X$
the holomorphic and anti-holomorphic tangent bundles of $X$. 
They are complex vector sub-bundles of $TX\otimes_{\R} \C$. 
The K\"ahler form $\theta$ induces a Riemannian metric $g^{TX}$ 
on $X$ defined by $g^{TX}:=\theta(\cdot,J\cdot)$. 
 
Let $\overline{\partial} ^{L^{p}}$ be the $\dbar$-operator 
acting on $L^{p}$ and 
$ \overline{\partial} ^{L^{p},*}$ its dual operator 
with respect to the metric $h=e^{-2\phi} h_0$ on $L$.  
Consider the Dirac and Laplacian-type operators 
\begin{equation}\label{eq:1.6}
D_{p}:=\sqrt{2}\big( \overline{\partial} ^{L^{p}}
+ \overline{\partial} ^{L^{p},*}\big) \qquad \text{and}\qquad
\square_{p}:=\frac{1}{2} D_{p}^{2}
= \overline{\partial} ^{L^{p}}\overline{\partial} ^{L^{p},*}
+\overline{\partial} ^{L^{p},*}\overline{\partial} ^{L^{p}}.
\end{equation}
They act on 
$\Omega^{0,\bullet}(X,L^p)$, the space
of the forms of bi-degree $(0,\cdot)$ with values in $L^p$.

Let $\nabla^L$ be the Chern connection on 
$(L, h= e^{-2\phi} h_{0})$ and $R^L=(\nabla^L)^2$ its curvature 
which is related to the first Chern 
form $\omega$ by
\begin{align}\label{eq:n2.8}
\omega={\sqrt{-1}\over 2\pi} R^L.
\end{align}
Let $\nabla^{TX}$ be the Levi-Civita connection on 
$(TX, g^{TX})$. It preserves $T^{(1,0)}X,\  T^{(0,1)}X$,
and its restriction to $T^{(1,0)}X$ is the Chern connection 
$\nabla^{T^{(1,0)}X}$. Let $\nabla^{\Lambda^{0,\bullet}}$
be the connection on $\Lambda(T^{*(0,1)}X)$ induced by 
$\nabla^{T^{(1,0)}X}$, and 
$\nabla^{\Lambda^{0,\bullet}\otimes L^p}$
the connection on $\Lambda(T^{*(0,1)}X)\otimes L^p$
induced by $\nabla^{\Lambda^{0,\bullet}}$ and $\nabla^L$. 
For $u\in T^{(1,0)}X$ and $v\in T^{(0,1)}X$, let 
$u^*\in T^{*(0,1)}X$ be the metric dual of 
$u$ with respect to $g^{TX}$, define the operator $c(\cdot)$ 
depending linearly on a vector in $T^{(1,0)}X\oplus T^{(0,1)}X$ 
by setting
\begin{align}\label{eq:n2.9}
c(u):=\sqrt{2} u^*\wedge \quad \text{and} \quad 
c(v):=\sqrt{2} i_{v} \ \!,
\end{align}
where $i$ denotes, as usual, the  contraction  operator.
Then by \cite[p.31]{MaMarinescu07}, for $\{e_{j}\}$ an orthonormal 
frame of $(TX, g^{TX})$, we have
\begin{align}\label{eq:n2.10}
D_{p}= \sum_{j} c(e_{j})
\nabla^{\Lambda^{0,\bullet}\otimes L^p}_{e_{j}}.
\end{align}

Denote by $K_X^*$ the anti-canonical bundle of $X$. 
 The curvature of $K_X^*$ with respect to the above Riemannian 
 metric is denoted by $R^{K^*_X}$. 
 Then $\sqrt{-1}R^{K^*_X}$ is the Ricci curvature of $(X, g^{TX})$. 
 Let $\{w_j\}_{j=1}^n$
be a local orthonormal frame of $T^{(1,0)}X$ 
with dual frame $\{w^j\}_{j=1}^n$. Set 
\begin{align}  \label{lm4.3}
&\omega_{d}:=-\sum_{l,m} R^L (w_l,\overline{w}_m)\,
\overline{w}^m\wedge \,i_{\overline{w}_l}.
\end{align}
Recall that ${\sqrt{-1}\over 2\pi} R^L=\omega\geq \zeta\theta$. 
Then $\omega_{d}$ is a section of  $\End(\Lambda (T^{*(0,1)}X))$ 
and $R^L $ acts as the derivative $\omega_{d}$ on 
$\Lambda (T^{*(0,1)}X)$.
By \cite[(1.4.63)]{MaMarinescu07} and using that
$\langle \Delta^{0,\bullet} s,s\rangle_{p\phi}\geq 0$, 
we obtain  
for $s\in \Omega^{0, \bullet}(X,L^{p})$ that
\begin{align}\label{dm1.10}
\| D_{p} s\|_{L^2(p\phi)}^2 
= 2\langle \square_p s,s\rangle_{p\phi}
\geq 
-2p \, \left\langle  \omega_{d}s,s\right\rangle_{p\phi}
+ 2\sum_{l,m}\left\langle  R^{K^*_{X}}(w_l,\overline{w}_m)\,
\overline{w}^m\wedge \,i_{\overline{w}_l}s,s\right\rangle_{p\phi}.
\end{align}
Now by (\ref{eq_strict_positivity}), \eqref{eq:n2.8}, \eqref{lm4.3} 
and some standard arguments, see the proof 
of \cite[Theorem 1.5.5]{MaMarinescu07}, 
there exists $\delta>0$, 
depending only on the Ricci curvature $R^{K_X^*}$,
such that if $\zeta p>\delta$, then
\begin{align}\label{1c4}
\spec(D_{p}^2)\subset \{0\}\cup [2\pi \zeta p, +\infty[.
\end{align}

Let $a_X$ denote the injectivity radius of $(X,\theta).$
For $0<\epsilon_0<a_X/4$, let $f_{\epsilon_0} : \R \to [0,1]$ 
be a smooth even function such that
 \begin{align} \label{lm4.19}
f_{\epsilon_0}(v) = \left \{ \begin{array}{ll}  1 \quad {\rm for}
\quad |v| \leqslant  \epsilon_0/2, \\
  0 \quad {\rm for} \quad |v| \geqslant  \epsilon_0.
\end{array}\right.
\end{align}
Set
 \begin{eqnarray} \label{bk2.8}
F_{\epsilon_0}(a) &: = &  
\Big(\int_{-\infty}^{+\infty}f_{\epsilon_0}(v) dv\Big)^{-1} 
\int_{-\infty}^{+\infty} e ^{i v\zeta a} f_{\epsilon_0}(v) dv  
\nonumber\\
& = & \Big(\int_{-\infty}^{+\infty}
f_{\epsilon_0}(\zeta^{-1}v) dv\Big)^{-1} 
\int_{-\infty}^{+\infty} e ^{i v a} f_{\epsilon_0}(\zeta^{-1}v) dv .
\end{eqnarray}
Then $F_{\epsilon_0}(a)$ lies in Schwartz space $\mathcal{S} (\R)$ 
and $F_{\epsilon_0}(0)=1$.

\begin{proposition}\label{0t3.0} 
Let $\delta>0$ verifying (\ref{1c4}).
Then, for all $l\in \N$, $0<\epsilon_0<a_X/4$ and $F_{\epsilon_0}$ 
as above, there exists $c>0$
such that for $p\geq 1$, $\delta/p< \zeta\leq 1$, $x,x'\in X$
\begin{align}\label{1c3}
\|F_{\epsilon_0}(D_{p})(x,x') - P_{p}(x,x')\|_{L^\infty(p\phi)}
\leq c \, |\omega|_{n}^{2n+2}
 \zeta^{-6n-3l- 6}  p^{-l}.
\end{align}
\end{proposition}

\begin{proof} For $a\in \R$, set
\begin{eqnarray}\label{0c2}
\phi_{p}(a) := 1_{[\sqrt{\zeta p}, +\infty[} (|a|) F_{\epsilon_0}(a).
\end{eqnarray}
By (\ref{1c4}) and (\ref{0c2}),  for $\zeta p>\delta$, we get
\begin{eqnarray}\label{0c3}
F_{\epsilon_0}(D_{p})-P_p = \phi_p(D_p).
\end{eqnarray}
By (\ref{bk2.8}), for any $m\in \N$ there exists $c>0$
such that for all $\zeta\in ]0,1[$,
\begin{equation}\label{1c9}
\sup_{a\in \R} |a|^{m} |F_{\epsilon_0} (a) | \leq c {\zeta}^{-m}.
\end{equation}
Thus, for any $m\in\N$ and $\zeta p>\delta$, we have
\begin{align}\label{1c5}
&\|(D_p)^m F_{\epsilon_0} (D_p)\|^{0,0}
:=\sup_{s\in  \Omega^{0, \bullet}(X,L^{p})\setminus\{0\}}  
{\| (D_p)^m F_{\epsilon_0} (D_p) s  \|_{L^2(p\phi)}\over 
\| s\|_{ L^2(p\phi) }}
\leq c \, \zeta^{-m}.
\end{align}

As $X$ is compact, there exists a finite set of points  
$a_i$, $1\leq i \leq r$, such that the family of balls
 $U_i := B^X(a_i,\epsilon_{0})$ of center $a_i$ and  radius 
 $\epsilon_0$, is a covering of $X$.
We identify the ball $B^{T_{a_i}X}(0,\epsilon_{0})$ in the
tangent space of 
$X$ at $a_i$ with the ball $B^{X} (a_i,\epsilon_{0})$
using  the exponential map. 
We then identify
$(TX)_Z, \Lambda (T^{*(0,1)}X)_{Z}$, $L^{p}_Z$ for $Z\in 
B^{T_{a_i}X}(0,\epsilon_{0})$
with $T_{a_i}X, \Lambda (T^{*(0,1)}X)_{a_i}$, $L^{p}_{a_i}$ 
 by parallel transport with respect to the connections 
 $\nabla ^{TX}$, $\nabla^{\Lambda^{0,\bullet}}$, 
 $\nabla^{L^{p}} $ along the curve
$\gamma_Z: [0,1]\ni u \mapsto \exp^X_{a_i} (uZ)$. 
Then $(L, h)|_{U_{i}}$ is identified as 
the trivial bundle $(L_{a_{i}}, h_{a_{i}})$.

Let $\{e_j\}_j$
be an orthonormal basis of $T_{a_i}X\simeq \R^{2n}$. 
Let $\wi{e}_j (Z)$ be 
the parallel transport of ${e}_j$ with respect to $\nabla^{TX}$ 
along the above curve.
Let $\Gamma^L , \Gamma ^{\Lambda^{0,\bullet}}$
be the corresponding connection forms of $\nabla^L $
and $\nabla^{\Lambda^{0,\bullet}}$ 
with respect to any fixed frame for
$L$ and $\Lambda (T^{*(0,1)}X)$ which is parallel along the curve 
$\gamma_Z$ under the trivialization on $U_i$.
Denote by  $\nabla_v$  the ordinary differentiation
 operator on $T_{a_i}X$ in the direction $v$. As we are working in 
 the K\"ahler case, by 
 \cite[Prop. 1.2.6, Th. 1.4.5, Rk. 1.4.8]{MaMarinescu07}, 
we can write on $U_i$
\begin{equation}\label{c10}
D_p =\sum_{j} c(\wi{e}_j) \Big( \nabla_{\wi{e}_j} 
+ p \Gamma^L (\wi{e}_j) 
+ \Gamma^{\Lambda^{0,\bullet}}(\wi{e}_j))  \Big).
\end{equation}
In fact, the last identity is a consequence of \eqref{eq:n2.10}.
Consider the radial vector field $\mathcal{R}= \sum_{j}Z_{j}e_{j}$.
By \cite[(1.2.32)]{MaMarinescu07}, the Lie derivative
$L_{\mathcal{R}}\Gamma^L$ is equal to  $i_{\mathcal{R}} R^L$. 
Therefore, we get the identity
\begin{align}\label{eq:n2.21}
\Gamma^L_{Z}= \int_{0}^{1} (i_{\mathcal{R}} R^L)_{tZ} dt,
\end{align}
which allows us to bound $\Gamma^L$. 

Let $\{ \varphi_i \}$ be a partition of unity subordinate to 
$\{U_i\}$. For $m\in \N$, we define a Sobolev norm on the $m$-th 
Sobolev space $H^m(X,\Lambda (T^{*(0,1)}X)\otimes L^{p})$ by
\begin{equation}\label{c11}
\| s\| _{H^m}^2 
= \sum_{i=1}^r \sum_{k=0}^m \sum_{j_1, \ldots, j_k=1} ^{2n}
\|\nabla_{e_{j_1}}\cdots  \nabla_{e_{j_k}}(\varphi _i s)
\|_{L^2}^2.
\end{equation}
Note that here we trivialize the line bundle $L$ using an unitary 
section; so the section $s$ above is identified with a function. 
Therefore, we drop the subscript $p\phi$ since this weight is 
already taken into account. 

By (\ref{c10}),
(\ref{eq:n2.21}) and \cite[(1.6.9)]{MaMarinescu07}, 
for $P$ a differential operator of order $m\in \N$ with
scalar principal symbol and with compact support
in $U_i$, we get
\begin{align}\label{lm4.17}
\begin{split}
\|Ps\|_{{H}^1} &  \leqslant  
c(\|D_pP s\|_{L^2} + p |\omega |_{0} \|Ps\|_{L^2}) 
\\
 & \leqslant c'\Big(\|P D_p s\|_{L^2} 
 + p \sum_{k=0}^{m} |\omega |_{k} \|s\|_{{H}^{m-k}}\Big),
\end{split}
\end{align}
for some constants $c,c'>0$. 
From (\ref{lm4.17}), we get by induction for (other) suitable 
constants $c,c'>0$ 
\begin{align}\label{c17}
\begin{split}
\|s\|_{H^{m+1}} & \leq  c \sum_{k=0}^{m+1} 
p^{m+1-k}\|D_p^k s\|_{L^2}
\prod_{\sum(k_r +1)=m-k+1} |\omega|_{k_r}  \\
&\leq c' \sum_{k=0}^{m+1}   
p^{m+1-k}\|D_p^k s\|_{L^2}|\omega|_{{m-k}}^{m-k+1}.
\end{split}
\end{align}
Note that for $k=m+1$ we set $|\omega|_{m-k}^{m-k+1}=1$. 

Let $Q$ be a differential operator of order $m'\in \N$ with
scalar principal symbol and with  compact support
in $U_j$. Using the identity
$$\big\langle D_p^{k} \phi_p(D_p)Q s, s'\big\rangle  
=\big\langle s,Q^ * \phi_p(D_p)  D_p^{k} s'\big\rangle, $$
we deduce from(\ref{c17}) with suitable sections instead of $s$ that 
\begin{align}\label{c19_bis}
\begin{split}
& \|P \phi_p( D_p) Q s\|_{L^2}\\
& \qquad \leq  c\sum_{k=0}^{m}\sum_{k'=0}^{m'} p^{m+m'-k-k'} 
\|D_p^k \phi_p( D_p)D_p^{k'}  s\|_{L^2}
|\omega|_{m-k-1}^{m-k} |\omega|_{m'-k'-1}^{m'-k'}\\
& \qquad =  c\sum_{k=0}^{m}\sum_{k'=0}^{m'} p^{m+m'-k-k'} 
\|D_p^{k+k'} \phi_p( D_p) s\|_{L^2} 
|\omega|_{m-k-1}^{m-k} |\omega|_{m'-k'-1}^{m'-k'}.
\end{split}
\end{align}
Note that the operators, considered in the last two lines, commute.
Thank to \eqref{1c4}, \eqref{0c2}, \eqref{0c3} and then  
\eqref{1c5}, if $0<\zeta\leq 1$ and  $\zeta p \geq \delta$,  
for any $q\in \N$, the main factor in the last line can be bounded 
using 
\begin{eqnarray}
 \|D_p^{k+k'} \phi_p( D_p)s\|_{L^2}   & 
 \leq &  (\zeta p)^{-q/2}\|D_p^{k+k'+q} \phi_p( D_p)s\|_{L^2}  \\
& \leq &  c (\zeta p)^{-q/2} \zeta^{-k-k'-q}   \|s\|_{L^2}. \nonumber
\end{eqnarray}
Take any $l>0$ and choose $q:=2(m+m'-k-k'+l)$. Then there exists 
$c_l>0$ such that  for $0<\zeta\leq 1$, $\zeta p \geq \delta$,
we have
\begin{align}\label{c19}\begin{split}
&\|P \phi_p( D_p) Q s\|_{L^2}\\
&\qquad  \leq  c\sum_{k=0}^{m}\sum_{k'=0}^{m'} p^{m+m'-k-k'}
 (\zeta p)^{-q/2} \zeta^{-k-q-k'}
|\omega|_{m-k-1}^{m-k} |\omega|_{m'-k'-1}^{m'-k'}\|s\|_{L^2} 
\\
&\qquad  \leq  c_{l} \zeta^{-3m-3m'-3l}  p^{-l} |\omega|_{m-1}^{m} 
|\omega|_{m'-1}^{m'} \|s\|_{L^2}.
 \end{split}
\end{align}
Finally, on $U_i\times U_j$, by using the standard Sobolev's
inequality  and  (\ref{0c3}),  we get  \eqref{1c3}.
Proposition \ref{0t3.0} follows.
\end{proof}

\begin{remark} \rm \label{rk_off_diag}
 By \eqref{bk2.8} and  the finite propagation speed of solutions
 of hyperbolic equations
\cite[Theorem D.2.1]{MaMarinescu07}, $F_{\epsilon_0}(D_p)(x, x')$ 
 only depends on the restriction of $D_p$ to\break
$B^X(x,\epsilon_{0}\zeta)$, and  
\begin{align}\label{absy3.30}
  F_{\epsilon_0}(D_p) (x,x')= 0 \quad  \text{  when   } 
  \quad  \dist(x, x') 
  \geqslant \epsilon_{0}\zeta.
\end{align}
\end{remark}

To get the uniform estimate of the Bergman kernels in terms 
of $\zeta,p$, 
we need an approach different from the use of 
the normal coordinates and the extension of connections on $L$
in \cite[\S 4.2]{DaiLiuMa06} and \cite[\S 4.1.3]{MaMarinescu07}.
Let $\psi: X\supset U\to V\subset \C^n$ be a holomorphic local 
chart such that $0\in V$ and $V$ is convex (by abuse of notation, 
we sometimes identify $U$ with $V$ and $x$ with $\psi(x)$).
Then, for any $x\in \frac{1}{2}V:=\{y\in \C^n: 2y\in V\}$,
we will use the holomorphic coordinates induced by $\psi$
and let $0<\epsilon_0\leq 1$ be such that 
$B(x, 4\epsilon_0)\subset V$
for any $x\in \frac{1}{2}V$. We choose $\epsilon_0$ smaller than 
$a_X/4$ in order to use the estimates given in the proof of 
Proposition \ref{0t3.0}.
Consider the holomorphic family of holomorphic local 
coordinates 
$\psi_{x}: \psi^{-1}(B(x, 4\epsilon_{0}))\to B(0, 4\epsilon_{0})$
for $x\in \frac{1}{2}V$ given by $\psi_{x}(y):= \psi(y)-x$.

Let $\sigma$ be a holomorphic frame of $L$ on $U$ and define 
the function 
$\varphi(Z)$ on $U$ by $|\sigma|_{\phi }^2(Z)=: e^{-2\varphi (Z)}$.
Consider the holomorphic family of holomorphic trivializations
of $L$ associated with the coordinates  
 $\psi_{x}$ and the frame $\sigma$.  
 These trivializations are given by 
 $\Psi_x:L|_{\psi^{-1}(B(x, 4\epsilon_{0}))}\to B(0, 4\epsilon_{0})
 \times \C$ with 
 $\Psi_x(y,v):=(\psi_{x}(y), v/\sigma(y))$ for $v$ a vector 
 in the fiber of $L$ over the point $y$. 

Consider a point $x_0\in{1\over 2} V$. Denote by 
$\varphi_{x_{0}}:=\varphi\circ \psi_{x_0}^{-1}$
the function $\varphi $ in local coordinates $\psi_{x_{0}}$. 
Denote also by $\varphi_{x_{0}}^{[1]}$ and 
$\varphi_{x_{0}}^{[2]}$ the first and second order 
Taylor expansions of $\varphi_{x_0}$, i.e.,
\begin{align}\label{1c14a&}\begin{split}
 \varphi_{x_{0}}^{[1]} (Z)&:= 
\sum_{j=1}^n \Big(\frac{\partial\varphi }{\partial z_j}(x_0) z_j +
\frac{\partial\varphi }{\partial \ov{z}_j}(x_0)\ov{z}_j\Big),\\
\varphi_{x_{0}}^{[2]} (Z)&:={\rm Re}\sum_{j,k=1}^n
\Big(\frac{\partial^2\varphi }{\partial z_j\partial z_k}(x_0) z_j z_{k}
+ \frac{\partial^2\varphi }{\partial z_j\partial \ov{z}_k}(x_0)  z_j
\ov{z}_k\Big),
\end{split}\end{align}
where we write $z=(z_1,\ldots,z_n)$ the complex coordinates 
of $Z$.

Let $\rho: \R\to [0,1]$ be a smooth even function such that
\begin{equation}\label{1c14}
\rho (t)=1  \  \  {\rm if} \  \   |t|<2\,; 
\quad \rho (t)=0 \   \   {\rm if} \  \ |t|>4\,.
\end{equation}
We denote in the sequel $X_0=\R^{2n}\simeq T_{x_0}X$ 
and equip $X_0$ with the metric 
$g^{TX_0}(Z):= g^{TX}(\rho(\epsilon_0^{-1}|Z|)Z)$.
Now  let $0<\epsilon<\epsilon_0$ and define
\begin{align}\label{absy3.32}
 \varphi_{\epsilon}(Z) := \rho(\epsilon^{-1}|Z|)   \varphi_{x_{0}}(Z) 
 + \big(1-\rho(\epsilon^{-1}|Z|)\big) 
\big(\varphi (x_0)+ \varphi_{x_{0}}^{[1]}(Z) 
+ \varphi_{x_{0}}^{[2]}(Z)\big).
\end{align}
Let $h^{L_{0}}_{\epsilon}$  be the metric on 
$L_0=X_0\times\C$  defined by 
\begin{align}\label{2c16}
|1|_{\varphi_\epsilon}^2(Z):= e^{-2\varphi_{\epsilon}(Z)}.
\end{align}
Here, as above, the subscript $\varphi_\epsilon$ informs the use 
of the weight $\varphi_\epsilon$.
Let $\nabla^{L_0}_{\epsilon}$ be the Chern
connection on $(L_0,h^{L_{0}}_{\epsilon})$ and
 $R_\epsilon^{L_{0}} $ be the curvature of  
 $\nabla^{L_0}_{\epsilon}$.

Then there exists a constant $A$ with $c |d\phi|_{2}^{-1}<A <1$ 
for $c>0$ such that when $\epsilon\leq A  \zeta$,
the following estimate holds  for every $x_0\in U$,
\begin{equation}\label{1c16}
\inf\Big\{\sqrt{-1} R^{L_0}_{\epsilon,Z}(u, J u)/|u|^2_{g^{TX_0}}: 
u\in T_Z X_0,\,Z\in X_0\Big\}\geqslant\frac{4}{5}\,  \zeta \,,
\end{equation}
because there exists $C>0$ such that 
for  $|Z|\leq 4\epsilon$,  $0\leq j\leq 2$, we have 
\begin{equation}\label{1c16a}
\Big|\varphi_{x_0}(Z)-\big(\varphi (x_0)
+ \varphi_{x_{0}}^{[1]}(Z)+ 
\varphi_{x_{0}}^{[2]}(Z) \big)\Big|_{\Cc^j}
\leq C |d\phi|_{2} |Z|^{3-j}.
\end{equation}

From now on, we take
\begin{align}\label{1c17}
\epsilon:= \epsilon_{0} A  \zeta. 
\end{align}
 Let $S_{x_0}$ be the unitary section of 
$(L_0,h^{L_{0}}_{\epsilon})$ that can be written as
$S_{x_0}=e^{-\tau} 1$ with  $\tau(x_0)=\varphi(x_0)$. 
So we have 
\begin{align}\label{1c18a}
\nabla^{L_0}_{Z} S_{x_0}
=i_{Z}(-d\tau -2 \partial \varphi_{\epsilon})S_{x_0}=0
\end{align}
and hence the function $\tau$ is given by
\begin{align}\label{1c17a}
\tau(Z)= \varphi (x_{0})
- 2 \int_0^1 (i_Z \, \partial \varphi_{\epsilon})_{tZ} dt.
\end{align}

Let 
\begin{align}\label{2c15}
D_p^{X_0}=\sqrt{2}(\ov{\partial}^{L_{0}^p}
+ \ov{\partial}^{L_{0}^p*}_{p\varphi_\epsilon})
\end{align}
be the Dolbeault operator on $X_0$ associated with the above data,
i.e., $\ov{\partial}^{L_{0}^p*}_{p\varphi_\epsilon}$ is the adjoint of 
 $\ov{\partial}^{L_{0}^p}$ with respect to the metrics
 $g^{TX_0}$ and $h^{L_0}_{\epsilon}$. Over the
ball $B(x_{0},2\epsilon)$, $D_p$ is just the restriction of
 $D_p^{X_0}$. 
 Now by \cite[Th. 1.4.7]{MaMarinescu07}, and observe that 
 the tensors associated with $g^{TX_0}$ do not depend
 on $\zeta$ and $\epsilon$,  as in (\ref{1c4}), 
 we get from (\ref{1c16}), the existence of a constant
$\delta>0$ such that for $\zeta p>\delta,$  
   \begin{align}\label{1u2}
\spec\,(D_p^{X_0})^2\subset 
\{0\} \cup \left[\, \zeta p,+\infty\right[.
\end{align}

 Using $S_{x_{0}}$, we get an isometry $L_0^p \simeq  \C$.
Let $P_p^{0}$ be the orthogonal projection from
$\Cc^\infty(X_0,L^{p}_0) \simeq \Cc^\infty (X_0,\C)$
on $\Ker D_p^{X_0}$. Let  $P_p^{0}(x,x')$
be the smooth kernel of $P_p^{0}$
with respect to the volume form $dv_{X_0}(x')$ induced by 
the metric $g^{TX_{0}}$.
We have the following result.

\begin{proposition} \label{p3.2} 
For all $l\in \N$, there exists $c>0$ such that 
for $\zeta p>\delta$, $x,x' \in B(x_{0},\epsilon)$,
\begin{align}\label{1c19}
\Big \|(P_p^{0}- P_p)(x,x')\Big \|_{\Cc^0}
\leq c \, (|d \phi |_{2}^{-1}\zeta)^{-6n-3l-6}  p^{-l} 
|\omega |_{{n}}^{2n+2}.
\end{align}
\end{proposition}
\begin{proof}  
First, we replace $f_{\epsilon_0}(v)$  in \eqref{lm4.19} 
by $f_{\epsilon_0}(v/A)$.
 By Remark \ref{rk_off_diag} and  \eqref{1c17}, for 
 $x,x' \in B(x_{0},\epsilon)$, we have 
$F_{\epsilon}(D_p) (x,x')= F_{\epsilon}(D_p^0) (x,x')$.
Now we have a version of Proposition \ref{0t3.0} for $P^0_p$ 
with $A\zeta$ instead of $\zeta$. Estimate  \eqref{1c19} follows.
\end{proof}

\section{Uniform estimate of the Bergman kernels}\label{s3.3}

We continue to use the notations introduced at the end of the last 
section. 
By Proposition \ref{p3.2}, in order to study the kernel  $P_p$, 
it suffices  to study the kernel $P_p^{0}$.
For this 
purpose, we will rescale the operator $(D_p^{X_0})^2$. 
Let $dv_{TX}$ be the Riemannian volume form of
$(T_{x_0}X, g^{T_{x_0}X})$. 
Let $\kappa (Z)$ be the smooth 
positive function defined by the equation
\begin{equation}\label{c22}
dv_{X_0}(Z) = \kappa (Z) dv_{TX}(Z),
\end{equation}
with $\kappa(0)=1$.

Let $\{e_j\}_{j=1}^{2n}$ be an oriented orthonormal basis of 
$T_{x_0}X$, and let $\{e ^j\}_{j=1}^{2n}$ be its dual basis.
They allow us to identify $X_0=\C^n$ with $\R^{2n}$ and we write 
$Z=(Z_1,\ldots,Z_{2n})$. 
If $\alpha = (\alpha_1,\cdots, \alpha_{2n})$ is a multi-index,
set $Z^\alpha := Z_1^{\alpha_1}\cdots Z_{2n}^{\alpha_{2n}}$.
 Denote by  $\nabla_U$ the ordinary differentiation
 operator on $T_{x_0}X$ in the direction $U$, 
and set $\partial_j:=\nabla _{e_j}$.
Set $t:=p^{-1/2}$.
For $s \in \Cc ^{\infty}(\R^{2n}, \C)$ and $Z\in \R^{2n}$, define
\begin{equation}\label{c27}
\begin{split}
(S_{t} s ) (Z) := & s (Z/t), \quad  
\nabla_{t} :=  tS_t^{-1}\kappa ^{1/2}
\nabla ^{L^{p}_0} \kappa ^{-1/2}S_t,\\
\Lc_{t} :=&  S_t^{-1} t^2 \kappa ^{1/2}\, (D_p^{X_0})^2\,
\kappa ^{-1/2}\,S_t. 
\end{split}
\end{equation}
Once we did the trivialization of $L_{0}$ on $X_{0}$, 
\eqref{c27} is well-defined for any $p\in \R$, $p\geq 1$.

The notations $\left \langle\,\cdot\,,\cdot\,\right\rangle_{0}$ 
and $\norm{\,\cdot\,}_{0}$ mean respectively 
the inner product and the $L^2$-norm on $\Cc ^\infty (X_0, \C)$
induced by $g^{TX_0}$. 
For $s\in \Cc^{\infty}_0(X_0, \C) $, set
\begin{align}\label{u0}
\begin{split}
&\|s\|_{t,0}^2:= \|s\|_{0}^2 
= \int_{\R^{2n}} |s(Z)|^2 dv_{TX}(Z),\\
&\| s \|_{t,m}^2 := \sum_{l=0}^m \sum_{j_1,\cdots, j_l=1}^{2n}
\| \nabla_{t,e_{j_1}} \cdots \nabla_{t,e_{j_l}} s\|_{t,0}^2.
\end{split}\end{align}
We then, for convenience, denote by $\langle s,s'\rangle_{t,0}$
the inner product on 
$\Cc^\infty(X_0, L_{x_0}^{\otimes p})$ corresponding to
the norm $\| \cdot\|_{t,0}.$ Let $H^m_t$ be the  
Sobolev space of order $m$ with norm $\| \cdot \|_{t,m}.$
Let $H^{-1}_t$ be the  Sobolev space of order $-1$  and
let $\| \cdot \|_{t,-1}$ be the norm on $H^{-1}_t$
  defined by
$\|s\|_{t,-1}:=\sup_{0\neq s'\in H^1_t}  
{|\langle s,s'\rangle_{t,0}|\over \| s'\|_{t,1}}.$ 
If $B:\ H^m_t\to H^{m'}_t$ is a bounded linear operator 
for $m,m'\in\Z,$ denote by $\|B\|^{m,m'}_t$ the norm  
of $B$  with respect to the norms  $\| \cdot\|_{t,m}$ and 
$\| \cdot\|_{t,m'}.$

Theorems \ref{tu1}, \ref{tu4}, \ref{tu6} and Proposition \ref{tu5} 
below are the analogues of \cite[Th. 4.1.9-4.1.14]{MaMarinescu07}
(cf. also \cite[Th. 4.7-4.10]{DaiLiuMa06}).
The emphasis here is the precise dependence of the involved 
constants on the curvature form $\omega $.

\begin{theorem}  \label{tu1} 
There exist $c_1, c_2, c_3>0$
such that for $t \in  ]0,1]$, $\zeta \in ]0,1]$, and 
$s,s' \in \Cc ^{\infty}_0(\R^{2n}, \C)$,
\begin{equation}\label{u1}
\begin{split}
& \left \langle \Lc_{t} s,s\right \rangle_{t,0} \geqslant 
c_1\|s\|_{t,1}^2 -c_2  |\omega |_{0} \|s\|_{t,0}^2 , \\
& |\left \langle \Lc_{t} s, s'\right \rangle_{t,0}|
 \leqslant c_3   |\omega |_{0} \|s\|_{t,1}\|s'\|_{t,1}.
\end{split}
\end{equation}
\end{theorem}
\begin{proof} By using the Lichnerowicz formula
    \cite[(4.1.33)]{MaMarinescu07}, the same arguments as in
    \cite[(4.1.38)--(4.1.39)]{MaMarinescu07} give the result.
 \end{proof}

Let $\delta_\zeta$ be the counter-clockwise oriented circle in $\C$
of center $0$ and radius $\zeta/2$.

\begin{theorem}\label{tu4} 
There exists $\delta>0$ such that  
the resolvent $(\lambda- \Lc_{t})^{-1}$ exists for 
all $\lambda \in \delta_\zeta$ and $t\in ]0, \sqrt{\zeta/\delta}]$. 
There exists $c>0$ such that for all
$t\in ]0,\sqrt{\zeta/\delta}]$,
$\lambda \in\delta_\zeta$, we have
\begin{equation}\label{ue2}
\| (\lambda- \Lc_{t})^{-1}\|^{0,0}_{t}
\leqslant 2\zeta^{-1},
\quad\| (\lambda- \Lc_{t})^{-1}\|^{-1,1}_{t}
\leqslant c\, |\omega |_{0}^2 \zeta^{-1} .
\end{equation}
\end{theorem}

\begin{proof} 
 By \eqref{1u2} and \eqref{c27}, we have    
\begin{align}\label{1u3}
\spec\,(\Lc_{t})\subset 
\{0\} \cup \left[\, \zeta ,+\infty\right[.
\end{align}
Thus, the resolvent $(\lambda- \Lc_{t})^{-1}$ exists for 
$\lambda \in \delta_\zeta$ and $t\in ]0, \sqrt{\zeta/\delta}]$,
and we get the  first inequality of (\ref{ue2}).

By \eqref{u1},
$(\lambda_0- \Lc_{t})^{-1}$ exists for $\lambda_0\in \R$, 
$\lambda_0\leqslant -2c_2 \,|\omega |_{0}$. Moreover, 
as $c_1\|s\|_{t,1}^2
\leq -\langle(\lambda_0- \Lc_{t}) s, s\rangle_{t,0}
\leq \|(\lambda_0- \Lc_{t}) s\|_{t,-1} \|s\|_{t,1}$, we have
\begin{equation} \label{ue8_bis}
\|(\lambda_0- \Lc_{t})^{-1}\|^{-1,1}_{t}
\leqslant {1\over c_1} \cdot
\end{equation}
On the other hand, we have
\begin{equation}\label{ue7}
(\lambda- \Lc_{t})^{-1}= (\lambda_0- \Lc_{t})^{-1}
- (\lambda-\lambda_0) (\lambda- \Lc_{t})^{-1}
(\lambda_0- \Lc_{t})^{-1}.
\end{equation}
Therefore,  for  $\lambda\in \delta_\zeta$, from
the first estimate in (\ref{ue2}) and (\ref{ue7}), we get
\begin{equation}\label{ue8}
\|(\lambda-\Lc_{t})^{-1}\|^{-1,0}_{t}  \leq
\frac{1}{c_1} \left(1+2 |\lambda-\lambda_0|\zeta^{-1}\right).
\end{equation}
In  \eqref{ue7}, we can interchange the last two factors. Then, 
applying \eqref{ue8_bis} and \eqref{ue8}  gives 
\begin{align}\label{ue9}
\|(\lambda-\Lc_{t})^{-1}\|^{-1,1}_{t}  
\leqslant  \frac{1}{c_1}+ \frac{|\lambda-\lambda_0|}{{c_1}^2} 
 \left(1+2 |\lambda-\lambda_0|\zeta^{-1}\right)
\leqslant c\, |\omega |_{0}^2\, \zeta^{-1}.
\end{align}
The theorem follows.
\end{proof}

\begin{proposition} \label{tu5} 
Take $m \in \N^*$. There exists 
$c>0$ such that for $t\in ]0,1]$, $Q_1, \ldots, Q_m$
$\in \{ \nabla_{t,e_j}, Z_j\}_{j=1}^{2n}$ and  
$s,s'\in\Cc^{\infty}_{0}(X_0, \C)$, 
\begin{equation}\label{ue11} 
\left |\left \langle [Q_1, [Q_2,\ldots 
[Q_m,  \Lc_{t}]\ldots]\,]s, s'\right
\rangle_{t,0} \right | \leqslant c |d\phi |_{m+1}^{\min (2,m)}
\|s\|_{t,1} \|s'\|_{t,1}. 
\end{equation}
\end{proposition}

\begin{proof} By \cite[(1.6.31)]{MaMarinescu07} and 
as in the proof of \cite[Proposition 1.6.9]{MaMarinescu07}, 
we know that 
$[Q_1, [Q_2,\ldots  [Q_m,  \Lc_{t}]\ldots]\,]$ has the same 
structure as $\Lc_{t}$ for $t\in [0,1]$.
More precisely, it has the form
\begin{align}\label{lm4.50}
\sum_{i,j} a_{ij}(t,tZ) \nabla_{ t,e_i}\nabla_{ t,e_j}
+\sum_{j}d_{j}(t,tZ) \nabla_{ t,e_j} + c(t,tZ),
\end{align}
where $a_{ij}(t,Z)$ and its derivatives in $Z$ are uniformly bounded,
$d_{j}(t,Z),  c(t,Z)$ and their first derivatives in $Z$
are bounded by $c |d\phi|_{m+1}^{\min (2,m)}$
for $Z\in \R ^{2n}$ and $t\in [0,1]$ and a constant $c>0$.
We then get estimate (\ref{ue11}).
\end{proof}

\begin{theorem}\label{tu6} For  
$Q_1, \ldots, Q_m\in \{\nabla_{t,e_j}, Z_j\}_{j=1}^{2n}$,
 there exists $c>0$ such that we have for 
 $t\in ]0, \sqrt{\zeta/\delta}]$, $\lambda \in \delta_\zeta $ and 
 $s\in \Cc ^\infty_{0} (X_0,\C)$,
\begin{multline}\label{ue12}
\|Q_1\cdots Q_m (\lambda-\Lc_{t})^{-1} s\|_{t,1}\\
\leqslant c  \sum_{k=0}^{m}\sum_{1\leq j_1<\cdots< j_k\leq m}
|d\phi|_{m-k+1}^{m-k}  (|\omega |_{0}^2\zeta^{-1})^{m-k+1} 
\|Q_{j_1}\cdots Q_{j_k} s\|_{t,0} .
\end{multline}
\end{theorem}
\begin{proof} For 
$Q_1, \ldots, Q_m\in \{\nabla_{t,e_j}, Z_j\}_{j=1}^{2n}$,
we can express $Q_1\cdots$ $Q_{m}(\lambda-\Lc_{t})^{-1}$ 
as the sum of $(\lambda-\Lc_{t})^{-1}Q_1\cdots$ $Q_{m}$ with
a linear combination  of operators of the type 
\begin{equation}\label{ue13} 
[Q_{j_1}, [Q_{j_2},\ldots  [Q_{j_{m_1}},
(\lambda-\Lc_{t})^{-1}]\ldots ]\,]\,
Q_{j_{m_1+1}} \cdots Q_{j_m},
\end{equation}
with $j_1<j_2\cdots <j_{m_1}$, $j_{m_1+1}< \cdots <j_m$.
The coefficients of this combination are bounded when $m$ 
is bounded. Let $\Sc_{t}$ be the family of operators 
$$[ Q_{j_1},[Q_{j_2},
\ldots [Q_{j_l},\Lc_{t}]\ldots ]\,] =-[ Q_{j_1},[Q_{j_2},
\ldots [Q_{j_l},\lambda-\Lc_{t}]\ldots ]\,] .$$  
Note that 
\begin{equation*}
[Q, (\lambda-\Lc_{t})^{-1}]= 
- (\lambda-\Lc_{t})^{-1} [Q, \lambda-\Lc_{t}]  (\lambda-\Lc_{t})^{-1}
=(\lambda-\Lc_{t})^{-1} [Q, \Lc_{t}]  (\lambda-\Lc_{t})^{-1},
\end{equation*}
thus by the recurrence on $m_1$ we know that every commutator 
$[Q_{j_1}, [Q_{j_2},\ldots  [Q_{j_{m_1}},(\lambda-\Lc_{t})^{-1}]
\ldots ]\,]$ is a linear
combination  of operators of the form 
\begin{equation}\label{ue14}
(\lambda-\Lc_{t})^{-1}S_1(\lambda-\Lc_{t})^{-1} S_2\cdots
S_{m_2}(\lambda-\Lc_{t})^{-1} 
\end{equation} with  $S_1, \ldots, S_{m_2} \in
\Sc_{t}$ and $m_2\leq m_1$. The coefficients of this combination 
are bounded when $m_1$ is bounded.

From Proposition \ref{tu5} we deduce that the 
 $\norm{\,\cdot\,}_{t}^{1,-1}$ norms of the operators\break
 $[ Q_{j_1},[Q_{j_2},\ldots [Q_{j_l},\Lc_{t}]\ldots ]\,]$ are
uniformly bounded from above by a constant times
$|d\phi |_{l+1}^l$.
Hence, by Theorem \ref{tu4}, the 
$\norm{\,\cdot\,}_{t}^{0,1}$ norm of the operator  \eqref{ue14} 
is bounded by a constant times
$$\zeta^{-m_{2}-1}
|\omega |_{0}^{2m_{2}+2}
\sum_{l_1+\cdots+l_{m_2}=m_1\atop l_1,\ldots,l_{m_2}\geq 1}
\prod_{j=1}^{m_2}|d\phi |_{l_{j}+1}^{\min(2,l_{j})}.$$
The theorem follows.
\end{proof}

Let 
$\mP_{t}:(\Cc^{\infty}(X_0,\C), \norm{\,\cdot\,}_0)\to
\Ker(\Lc_{t})$ be the orthogonal projection corresponding to  
the norm $\|\cdot\|_{t,0}$ given in  (\ref{u0}).
Let $\mP_{t}(Z,Z')$, 
(with $Z,Z'\in X_0$) be the smooth kernel of $\mP_{t}$ 
 with respect to $dv_{TX}(Z')$. Note that
$\Lc_{t}$ is a family of differential operators on $T_{x_0}X$ 
with coefficients in $\C$. 
Let $\pi : TX\times_{X} TX \to X$ be the
natural projection from the fiberwise product of $TX$ 
with itself on $X$. 
We can view  $\mP_{t}(Z,Z')$ as  smooth 
functions over $TX\times_{X} TX$
by identifying a section $F\in \Cc^\infty (TX\times_{X}TX, \C)$
with the family $(F_{x_0})_{x_0\in X}$, where
$F_{x_0}:=F|_{\pi^{-1}(x_0)}$.
In the following result we adapt 
\cite[Theorem 4.1.24]{MaMarinescu07} to the present situation.

\begin{theorem}\label{tue8}
 For any $r\in\N$, $\sigma>0$, there exists $c>0$,
 such that for $t\in ]0,\sqrt{\zeta/\delta}]$ and 
 $Z,Z'\in T_{x_0}X$ with $|Z|,|Z'|\leqslant \sigma$,
 \begin{align}\label{ue15}
&\Big\|\frac{\partial^{r}}{\partial t^{r}}
\mP_{t}\left (Z, Z'\right )\Big \|_{\Cc^{0}(X)}   
\leqslant c \zeta^{-2n-4r-2}
 |d\phi  |_{2r+n+1}^{4r+2n} |\omega |_{0}^{8r+4n+4}
 |d\phi |_{n+2}^{2n+2}.
\end{align}
\end{theorem}
\begin{proof} 
By (\ref{1u3}), for every $k\in \N^*$,
\begin{align}\label{1ue15}
&\mP_{t}= \frac{1}{2\pi \sqrt{-1}}  
\int_{\delta_\zeta} \lambda^{k-1} 
(\lambda - \Lc_{t})^{-k} d \lambda.
\end{align}
For $m\in \N$, let $\mQ ^m$ be the set of operators
 $\nabla_{t,e_{i_1}}\cdots    \nabla_{t,e_{i_j}}$ with $j\leqslant m$.
We apply Theorem \ref{tu6} to $m-1$ operators $Q_2,\ldots, Q_m$ 
instead of $m$ operators.  We deduce that for $l,m\in \N^*$ with 
$l\geq m$, and $Q=(Q_1,\ldots,Q_m)\in \mQ ^{m}$,  
there are $c,c'>0$ such that for 
$t\in ]0,\sqrt{\zeta/\delta}]$, $\zeta\in [0,1]$,
$s\in \Cc ^\infty_{0} (X_0,\C)$, and $\lambda\in\delta_\zeta $
\begin{eqnarray}\label{ue18}
\lefteqn{\| Q_1\cdots Q_m (\lambda-\Lc_{t})^{-l} s\|_{t,0}
\  \leqslant \  c \| Q_2\cdots Q_m (\lambda-\Lc_{t})^{-l} s\|_{t,1} }
\nonumber\\
&\qquad \leqslant & c'\sum_{k=0}^{m-1}
\sum_{1<i_1<\cdots< i_k\leq m}
|d\phi |_{{m-k}}^{m-k-1} (|\omega |_{0}^2\zeta^{-1} )^{m-k}
\|Q_{i_1}\cdots Q_{i_k}
(\lambda-\Lc_{t})^{-l+1} s\|_{t,0}.
\end{eqnarray}
Then, by induction and using \eqref{ue2}, we get 
\begin{align}\label{ue18_bis}
\| Q_1\cdots Q_m (\lambda-\Lc_{t})^{-l} s\|_{t,0} 
\leqslant   c \zeta^{-m-l+1} |d\phi  |_{{m}}^{m-1}
|\omega |_{0}^{2m} \|s\|_{t,0} \, .
\end{align}
As $\Lc_{t}$ is symmetric, we
can consider the adjoint of the operator in \eqref{ue18_bis} and 
get for $Q'=(Q_1',\ldots, Q_{m'}')\in \mQ^{m'}$, 
\begin{align}\label{ue18_ter}
&\|  (\lambda-\Lc_{t})^{-l} QÕ_1\ldots Q'_{m'} s\|_{t,0}
\leqslant   c \zeta^{-m'-l+1} |d\phi  |_{{m'}}^{m'-1}
|\omega |_{0}^{2m'} \|s\|_{t,0} \, .
\end{align}

Note that for $m=0$ and $l\in\N$ we also have 
$\|(\lambda-\Lc_{t})^{-l} s\|_{t,0} \leq c \zeta^{-l}\|s\|_{t,0}$.
Thus, for $Q\in \mQ^{m}, Q'\in \mQ^{m'}$ with $m,m'>0$,
 by taking $k= m+m'$, we get 
    \begin{align}\label{ue21}
\begin{split}
\|Q \mP_{t}Q' \|^{0,0}_{t} & \leqslant   \frac{1}{2\pi } 
\int_{\delta_\zeta } |\lambda|^{m+m'-1}
\|Q(\lambda - \Lc_{t})^{-m-m'}Q'\|^{0,0}_{t} |d \lambda|  \\
& \leqslant  c \,  |d\phi  |_{{m}}^{m-1}|\omega |_{0}^{2m}
  |d\phi  |_{{m'}}^{m'-1}|\omega |_{0}^{2m'}
  \zeta^{-2m-2m'+2}  \zeta^{m+m'}\\
  & =  c \,  |d\phi  |_{{m}}^{m-1}|\omega |_{0}^{2m}
  |d\phi  |_{{m'}}^{m'-1}|\omega |_{0}^{2m'}
  \zeta^{-m-m'+2}.
\end{split}
\end{align}

By \cite[Lemma 1.2.4]{MaMarinescu07}, (\ref{1c18a}),
(\ref{1c17a}) and (\ref{c27}) ,
on $B^{T_{x_0}X}(0, \epsilon/t)$,
 \begin{align}\label{0c40}
 &\nabla_{t, e_i}|_{Z} = \nabla_{e_i} + 
 \frac{1}{2}R^{L}_{x_{0}}(Z,e_{i}) 
 + O(t |Z|^{2}) |d\phi |_{2}.
 \end{align}
Let $|\,\cdot\,|_{(\sigma),m}$ denote the usual Sobolev norm on
 $\Cc^\infty(B^{T_{x_0}X}(0,\sigma+1),
 \C)$ induced by 
 the volume form $dv_{TX}(Z)$ as in (\ref{u0}).  
 Observe that by
(\ref{u0}),  (\ref{0c40}),  for $m>0$, 
 there exists $c>0$ such that for $s\in
 \Cc^\infty (X_0, \C)$ with $\supp (s) \subset B(0,\sigma +1)$, 
 \begin{align}\label{1ue21}
 &\frac{1}{c  |d\phi |_{m+1}^m}   \|s\|_{t,m}
 \leqslant |s|_{(\sigma),m}
 \leqslant c  |d\phi |_{m+1}^m\|s\|_{t,m}.
 \end{align}

Now, we want to estimate  $Q_Z Q'_{Z'} \mP_{t}(Z,Z')$ using 
the standard Sobolev's inequality for $Q \in \mQ ^{m}$ and  
$Q' \in \mQ ^{m'}$. If we define 
$S:=Q\mP_tQ'$ then we have for $|Z|, |Z'|\leq \sigma$
\begin{multline} \label{1ue23_bis}
|Q_Z Q'_{Z'} \mP_{t}(Z,Z') | \leq c  \sup
 \Big\{ \Big\| \frac{\partial^{|\alpha|} }{\partial {Z}^\alpha} 
 S \frac{\partial^{|\alpha'|} s}{\partial {Z'}^{\alpha'}}  
 \Big\|_{(\sigma),n+1}, \\
 \|s\|_{L^2}=1, \supp(s)\subset B(0,\sigma+1), |\alpha|,|\alpha'|
 \leq n+1 \Big\}.
\end{multline}
Hence, by 
\eqref{1ue21}, applied twice to $n+1$ instead of $m$, 
and also \eqref{ue21}, applied to $m+n+1, m'+n+1$ instead of 
$m,m'$,  we get
\begin{equation}\label{1ue23}
\begin{split}
& \sup_{|Z|,|Z'|\leqslant \sigma}| Q_Z Q'_{Z'} \mP_{t}(Z,Z') | \\
& \qquad \leqslant c'\, |d\phi |_{{m+n+1}}^{m+n}  
|\omega |_{0}^{2m+2n+2}  
|d\phi|_{m'+n+1}^{m'+n}  
|\omega |_{0}^{2m'+2n+2}  
|d\phi |_{n+2}^{2n+2}  \zeta^{-m-m'-2n}\,.
\end{split}
\end{equation}
By \eqref{0c40} and \eqref{1ue23} for $m=m'=0$,
 estimate \eqref{ue15} holds for $r=0$.

Consider now $r\geq 1$. Set
\begin{equation}\label{ue29}
I_{k,r} := \Big \{ (\bk,\br)=\{(k_i,r_i)\}_{i=0}^j: 
\sum_{i=0}^j k_i =k+j, \sum_{i=1}^j r_i =r,\, \, 
k_i, r_i \in \N^*\Big \}.
\end{equation}
Then there exist $a^{\bk}_{\br} \in \R$ such that
\begin{equation}\label{ue30}
\begin{split}
& A^{\bk}_{\br}  (\lambda,t) = (\lambda-\Lc_{t})^{-k_0}
\frac{\partial^{r_1}\Lc_{t}}{\partial t^{r_1}}
(\lambda-\Lc_{t})^{-k_1}\cdots
\frac{\partial^{r_j}\Lc_{t}}{\partial t^{r_j}}  
(\lambda-\Lc_{t})^{-k_j},\\
& \frac{\partial^{r}}{\partial t^{r}}
(\lambda-\Lc_{t})^{-k}=
\sum_{(\bk,\br)\in I_{k,r} }
 a ^{\bk}_{\br}  A ^{\bk}_{\br}  (\lambda,t).
\end{split}
\end{equation}
Set $g_{ij}(Z):= \langle\frac{\partial}{\partial Z_i}, 
\frac{\partial}{\partial Z_j}\rangle_Z$,
and $(g^{ij})$ the inverse matrix of $(g_{ij})$.
Note that $\frac{\partial^{u}}{\partial t^{u}}(g^{ij}(tZ))$,
$\frac{\partial^{u}}{\partial t^{u}}(\nabla_{ t,e_i} 
- \frac{1}{t}\Gamma^L (tZ))$
are functions which do not depend on $\zeta$,
and $\frac{\partial^{u}}{\partial t^{u}} R^L_\zeta(tZ)$,
$\frac{\partial^{u}}{\partial t^{u}} \Gamma^L (tZ)$
are functions of type $d'(tZ)Z^\beta$,
and $\nabla_{e_{j_1}}\cdots \nabla_{e_{j_l}} d'(tZ)$
is uniformly controlled by $|d\phi  |_{l+u+1}$.

We handle now the operator $A ^{\bk}_{\br}(\lambda,t)Q'$. 
We will move first all the terms $Z^\beta$ in $d'(tZ)Z^\beta$
(defined above) to the right hand side of this operator.  To do so,
we always use the commutator trick as in the proof of 
\cite[Theorem 1.6.10]{MaMarinescu07}, i.e., each time,
we perform only the commutation with $Z_i$
(not directly with $Z^\beta$ with $|\beta|>1$).
Then $A ^{\bk}_{\br}(\lambda,t)Q'$ is as the form
$\sum_{|\beta|\leqslant 2r} \Lc_{\beta,t} Q''_\beta  Z^\beta$, 
and $Q''_\beta$
 is obtained from $Q'$ and its commutation with $Z^\beta$.
 Observe that $[Z_i,\Lc_t]$ is a first order differential operator
 and $[Z_{j_1},[Z_{j_2},\Lc_t]]= g^{j_1 j_2}(tZ)$
 is a bounded function. Therefore, $\Lc_{\beta,t}$ is 
 a linear combination of operators of the form
\begin{align}\label{abk2.38}
 (\lambda-\Lc_t)^{-k'_0}S_1(\lambda-\Lc_t)^{-k'_1}S_2
\cdots S_{l'}(\lambda-\Lc_t)^{-k'_{l'}},
\end{align}
with  $S_i\in \{ a(tZ)\nabla_{e_{t, j_1}} \nabla_{e_{t, j_2}} , 
d_{j_1}(tZ)\nabla_{e_{t, j_1}}, d'_\zeta(tZ)\}$
and the number of $\nabla_{e_{t, j_1}}$ in all $\{S_i\}_i$ is
 less than $\sum_i r_i +2j=r+2j$. As $k>2(r+1)+m+m'$, 
 we can split the above operator into two parts
as in \cite[(4.1.51)]{MaMarinescu07} and use the fact that 
 the term $\nabla_{e_{t, j}} (\lambda-\Lc_{t})^{-l_1}$ 
 will contribute $\zeta^{-l_1}$. Similarly to
 (\ref{ue18}), we get that
$A^{\bk}_{\br}(\lambda,t)$ is well defined and
for $m, m'\in \N$, $k>2(r+1)+ m+m'$, 
$Q\in \mQ^{m},Q'\in \mQ^{m'}$,
there exists $c>0$ such that for $\lambda\in \delta_\zeta$ and 
$t\in ]0,\sqrt{\zeta/\delta}]$,
    \begin{align}\label{1ue30}
\begin{split}
&\left\|Q  A ^{\bk}_{\br}(\lambda,t)Q' s\right\|_{t,0} \\
&\quad \leqslant c|d\phi |_{{m+2r}}^{m+2r-1} 
|\omega |_{0}^{2m+4r}  
|d\phi  |_{{m'+2r}}^{m'+2r-1}  |\omega |_{0}^{2m'+4r} \
\zeta^{-\sum_{i=0}^{j}k_{i} -m-m'-3r} 
 \sum_{|\beta|\leqslant 2r} \|Z^\beta s\|_{t,0} \\
& \quad\leq  c |d\phi  |_{{m+2r}}^{m+2r-1}  
|\omega |_{0}^{2m+4r}  
|d\phi  |_{{m'+2r}}^{m'+2r-1} 
|\omega |_{0}^{2m'+4r} \zeta^{-k -m-m'-4r} 
\sum_{|\beta|\leqslant 2r} \|Z^\beta s\|_{t,0}.
\end{split}\end{align}

 By  (\ref{1ue15}), (\ref{ue30}) and (\ref{1ue30}), as in  
 (\ref{ue21}), for $m,r\in \N$,  $Q\in \mQ^m$ and $Q'\in \mQ^{m'}$,
there exists $c>0$ such that for $t\in ]0,\sqrt{\zeta/\delta}]$ and
 $s\in  \Cc^{\infty}_0(X_0, \C) $,
 \begin{multline}\label{1ue31}
 \left\|Q  \frac{\partial^{r}}{\partial t^{r}} 
 \mP_{t} Q' s\right\|_{t,0}  
 \leqslant  c  |d\phi  |_{{m+2r}}^{m+2r-1}  
  |d\phi |_{{m'+2r}}^{m'+2r-1} 
  |\omega |_{0}^{2m+2m'+ 8r} \zeta^{-m-m'-4r} 
\sum_{|\beta|\leqslant 2r} \|Z^\beta s\|_{t,0}.
 \end{multline}

Finally, \eqref{1ue21} and \eqref{1ue31}  together
 with Sobolev's inequalities imply for 
 $|Z|, |Z'|\leqslant \sigma$,
 \begin{equation}\label{ue23}
 \sup_{|Z|,|Z'|\leqslant \sigma}\left| 
 \frac{\partial^{r}}{\partial t^{r}} \mP_{t}(Z,Z') \right| 
 \leqslant c |d\phi |_{{2r+n+1}}^{2n+4r}
 |\omega |_{0}^{2(2n+2+4r)} 
 |d\phi  |_{n+2}^{2n+2} \zeta^{-2n-4r-2} \,.
 \end{equation}
 This ends the proof of the theorem.
\end{proof} 
     
    For $k$ big enough, set
    \begin{equation}\label{ue31}
    \begin{split}
    & F_{r}:= \frac{1}{2\pi \sqrt{-1} \, r! } 
      \int_{\delta_\zeta}
    \lambda ^{k-1}   \sum_{(\bk,\br)\in I_{k,r} }
     a ^{\bk}_{\br}  A ^{\bk}_{\br}  (\lambda,0)d \lambda .
    \end{split}
    \end{equation}
 Let $F_{r}(Z,Z')\in \Cc^\infty (TX\times_{X}TX,\C)$ 
be the smooth kernel of $F_{r}$  with respect to $dv_{TX}(Z')$.
   
\begin{theorem} \label{tue14} 
For all $j\in \N$, $\sigma>0$, 
there exists $c>0$  such that for  $t\in ]0,\sqrt{\zeta/\delta}]$ 
and $Z,Z'\in T_{x_0}X$,  $|Z|, |Z'|\leqslant \sigma$,
we have 
\begin{eqnarray}\label{0ue45}
\ \ \ \qquad  \Big\|\Big (\mP_{t} 
 - \sum_{r=0}^j  F_{r}t^r\Big ) (Z,Z') \Big\|_{\Cc ^{0}(X)} 
 & \leqslant &  c  |d\phi |_{{2j+n+3}}^{2(2j+n+2)} 
  |\omega |_{0}^{2(4j+2n+6)}  |d\phi |_{n+2}^{2n+2}
  \zeta^{-4j-2n- 6}  t^{j+1} .
\end{eqnarray}
\end{theorem}
\begin{proof} 
  By \cite[(4.1.69)]{MaMarinescu07},	we have
     \begin{equation}\label{0ue47}
    \frac{1}{r!}  \frac{\partial^{r}}{\partial t^{r}} \mP_{t} 
    \Big|_{t=0} = F_{r}\,.
    \end{equation}
    Recall that the Taylor expansion with integral rest of 
    a function $G\in\Cc^{j+1}([0,1])$ is
    \begin{equation}\label{0ue46}
    G(t)- \sum_{r=0}^j \frac{1}{r!} 
    \frac{\partial ^r G}{\partial t^r}(0) t^r
    = \frac{1}{j!}\int_0^t (t-t_0)^j 
    \frac{\partial ^{j+1} G}{\partial t^{j+1} }(t_0) dt_0\,,
    \quad t\in[0,1]\,.
    \end{equation}
    Theorem \ref{tue8} and \eqref{0ue47} show that 
    the estimate \eqref{ue15} holds if we replace  
    $\frac{1}{r!}\frac{\partial^{r}}{\partial
    t^{r}}\mP_{t}$ with $F_{r}$.
    Using this new estimate 
    together with  \eqref{0ue46} and \eqref{ue15}, 
    we obtain \eqref{0ue45}.
    \end{proof}

Let $ \mP$ be the orthogonal  projection from 
$L^2(X_0,  \C)$ 
onto $\Ker(\Lc_0),$ and let $ \mP(Z,Z')$  be the 
smooth kernel of  $ \mP$ with respect to $dv_{TX}(Z').$ 
Then $ \mP(Z,Z')$ is  the Bergman kernel of $\Lc_0.$  
By \cite[(4.1.84)]{MaMarinescu07}, if we choose $\{w_{j}\}$ to be 
an orthonormal basis of $T^{(1,0)}_{x_{0}} X$ such that 
$\dot{R}^L_{x_{0}}= {\rm diag}(a_{1},\cdots, a_{n})
\in \End(T^{(1,0)}_{x_{0}} X)$ 
with $\left\langle  \dot{R}^L_{x_{0}} W,\overline{Y}\right\rangle
= R^L(W,\overline{Y})$ for $W,Y\in T^{(1,0)}_{x_{0}} X$,
then 
\begin{equation}\label{bk2.71}
\begin{split}
\mP(Z,Z') &=\prod_{i=1}^n
\frac{a_i}{2\pi}\:\:\exp\Big(-\frac{1}{4}\sum_i
a_i\big(|z_i|^2+|z^{\prime}_i|^2 -2z_i\overline{z}_i'\big)\Big).
\end{split}
\end{equation}
The following result was established in 
\cite[Theorem 4.1.21]{MaMarinescu07}.
  
    \begin{theorem} \label{0t3.6} 
    There exist polynomials 
$J_r(Z,Z')$ in $Z,Z'$ with the same parity as $r$ and 
$\deg J_r(Z,Z')\leqslant 3r$, whose coefficients  
    are polynomials in $R^{TX}$  {\rm(} resp. $R^L${\rm)}  
    and their derivatives of order $\leqslant r-2$ 
    {\rm(}resp. $\leqslant r${\rm)}, 
    and reciprocals of linear combinations of eigenvalues of 
 $R^L$ at $x_0$\,, such that    
    \begin{align}\label{c86}
    &F_{r}(Z,Z')= J_{r}(Z,Z')\mP(Z,Z').
    \end{align}
    Moreover, we have
    \begin{equation}\label{1c52}   \begin{split}
    &J_0=1 \quad \text{and} \quad F_{0}=\mP.   
     \end{split} 
\end{equation}
    \end{theorem}
 
Owing to \eqref{c22}, \eqref{c27},
as in \cite[(4.1.96)]{MaMarinescu07}, we have
\begin{align}\label{0c53} P_p^0(Z,Z') =
t^{-2n}\,\kappa ^{-1/2}(Z)\,\mP_{t}(Z/t, Z'/t)\, 
\kappa ^{-1/2}(Z')\,,\, \,  \text{for all $Z,Z'\in\R^{2n}$}.
\end{align}

From Theorems \ref{tue14} and \ref{0t3.6} and \eqref{0c53},
we get the following near-diagonal expansion
of the Bergman kernels. Recall that we are working
with $t=p^{-1/2}$. 

    \begin{theorem} \label{t3.8} 
    For every $j\in \N$,
     there exists $c>0$  such that the estimate 
     \begin{multline}\label{1c53}
    \Big\arrowvert 
    \Big( \frac{1}{p^n}P_p^0 (Z,Z')
     - \sum_{r=0}^j F_{r} (\sqrt{p}Z,\sqrt{p}Z')
    \kappa ^{-1/2}(Z)\kappa^{-1/2}(Z')
    p^{-r/2}\Big) \Big\arrowvert\\
     \leqslant c  |d\phi  |_{{2j+n+3}}^{2(2j+n+2)} 
     |\omega |_{0}^{2(2n+4j+6)} 
      |d\phi  |_{n+2}^{2n+2} p^{-(j+1)/2}
  \zeta^{-2n-4j-6} 
    \end{multline}
    holds for all $0<\zeta\leq 1$, $\zeta p>\delta$,
    and all $Z,Z'\in T_{x_0}X$  with 
    $|Z|,|Z'| \leq \min\{\sigma/\sqrt{p}, \epsilon\}$.
    \end{theorem}

\noindent
{\bf End of the proof of Theorem \ref{tue4}.}
We apply  Theorem \ref{t3.8} to $Z=Z'=0$ and  $j=1$. 
Note that $F_1(0,0)=0$ because the function $F_1$ is odd. 
By \eqref{bk2.71}, $\mP(0,0)= \omega(x_0)^n/\theta(x_0)^n$.
So from \eqref{1c53}, we get 
  \begin{align}\label{1c56}
  \Big\| \frac{1}{p^n}P_p^0 (0,0)
  - {\omega(x_0)^n\over \theta(x_0)^n}  \Big\|_{\Cc ^{0}(X)}
\leqslant c |d\phi  |_{{n+5}}^{2n+8} |\omega |_{0}
 ^{4n+20}    |d\phi  |_{n+2}^{2n+2} \zeta^{-2n-10} p^{-1}.
       \end{align}
We then deduce the result form  
Propositions \ref{0t3.0}, \ref{p3.2} and (\ref{1c56}).
\hfill $\square$

\begin{remark} \label{Rem3.10}  \rm
Assume now $\phi\in \Cc^{n+2k+6}$. 
Then by the usual $\Cc^k$-norm on each 
$U_{j}$ and Sobolev embedding theorem, from \eqref{c19},
we get
\begin{align}\label{1c3b}
\|F_{\epsilon_0}(D_{p})(x,x') - P_{p}(x,x')\|_{\Cc^k}
\leq c \, |\omega|_{n+k}^{2n+2+2k}
 \zeta^{-6n-3l- 6-6k}  p^{-l}.
\end{align}
Note that $\nabla^{L^p}=  d + p \Gamma^L$ cf. \eqref{c10},
thus if we use the $\Cc^k$-norm induced by $\nabla^{L^p}$,
then we get 
\begin{align}\label{1c3c}\begin{split}
\|F_{\epsilon_0}(D_{p})(x,x') &- P_{p}(x,x')\|_{\Cc^k(X\times X)}\\
&\leq c \sum_{r=0}^k \, |\omega|_{n+m}^{2n+2+2r}
 \zeta^{-6n-3(k-r+1) - 6-6k_{1}} p^{-k-r-1} 
 |\omega|_{k-r}^{k-r}  p^{k-r}\\
& \leq c \, |\omega|_{n+k}^{2n+2+2k}
 \zeta^{-6n- 9-6k} p^{-1}.
\end{split}\end{align}
In the same way as \eqref{1c19} and above, we get 
\begin{align}\label{1c19b}
\Big \|(P_p^{0}- P_p)(x,x')\Big \|_{\Cc^k(X\times X)}
\leq c \, (|d \phi |_{2}^{-1}\zeta)^{-6n-6k -9}  p^{-1} 
 |\omega|_{n+k}^{2n+2+2k}.
\end{align}
Combining \cite[(4.1.64)]{MaMarinescu07}
and the argument for \eqref{ue15}, we get
 \begin{align}\label{ue15b}
&\Big\|\frac{\partial^{r}}{\partial t^{r}}
\mP_{t}\left (Z, Z'\right )\Big \|_{\Cc^{m'}}   
\leqslant c \zeta^{-2n-4(r+m')-2}
 |d\phi  |_{2(r+m')+n+1}^{4(r+m')+2n} |\omega |_{0}^{8(r+m')+4n+4}
 |d\phi |_{n+2}^{2n+2},
\end{align}
here $\Cc^{m'}$ is the ususal $\Cc^{m'}$-norm for 
the parameter $x_{0}$.

Thus we get an extension of \eqref{smp1.6},
\begin{align}\label{smp1.6b}
\begin{split}
\Big\|p^{-n}P_{p}(x,x)- {\omega(x)^n\over\theta(x)^n}
\Big\|_{\Cc^k}
\leqslant & c\, |d\phi|_{{n+2k+5}}^{2n+4k+8}
|\omega|_{0}^{4n+8k+20}
|d\phi|_{n+2}^{2n+2} \zeta^{-2n-4k -10} p^{-1}  \\
& +c |\omega|_{{n+k}}^{2n+2k+2} (|d\phi|_{2}\zeta^{-1})^{6n+9+6k}
p^{-1} .
\end{split}\end{align}
\end{remark}

\begin{remark} \rm
Let $\phi$ be a function of class $\Cc^\alpha$, $0<\alpha\leq 1$, 
which is $\omega_0$-p.s.h., i.e., $\ddc \phi+\omega_0\geq 0$. 
For each $0<\zeta\leq 1$, we can find a smooth $\omega_0$-p.s.h. 
function $\phi_\zeta$ such that 
$\|\phi_\zeta\|_{\Cc^k} \leq c \zeta^{-k+\alpha}$ and 
$\ddc\phi_\zeta+\omega_0\geq \zeta\omega_0$, 
see \cite{DMN15}. As mentioned in Introduction, we can study 
$\phi$ by applying our results to $\phi_\zeta$. Some steps in 
the proof of our estimates can be strengthened using 
$\|\phi_\zeta\|_{\Cc^k} \leq c \zeta^{-k+\alpha}$ for each 
$0\leq k\leq n+6$ instead of using only the $\Cc^{n+6}$-norm.
\end{remark}

\end{document}